\documentclass{article}

\usepackage{amsmath}
\usepackage{amssymb}
\usepackage{amsthm}
\usepackage{latexsym}
\usepackage{bbold}
\usepackage{stmaryrd}
\usepackage{dsfont}
\usepackage{hyperref}
\usepackage{enumerate}
\usepackage{color}
\usepackage[utf8]{inputenc}
\usepackage[english]{babel}
\usepackage{amsfonts}%
\usepackage{graphicx}

\newtheorem{theorem}{Theorem}
\newtheorem{corollary}[theorem]{Corollary}
\newtheorem{lemma}[theorem]{Lemma}

\newtheorem{problem}[theorem]{Problem}

\theoremstyle{definition}

\newtheorem{remark}[theorem]{Remark}

\newcommand{\integers}{\ensuremath{\mathds{Z}}} 

\newcommand{\comment}[1]{}

\newcommand{\pivot}[1]{\ensuremath{\pi_{{#1}}}} 
\newcommand{\tcz}{TC\textsuperscript{0}}

{}
{}
{}

\title{Low-complexity computations for nilpotent subgroup problems}
\author{Jeremy Macdonald\footnote{Concordia University, jeremy.macdonald@concordia.ca}, Alexei Miasnikov\footnote{Stevens Institute of Technology,  Alexei.Miasnikov@stevens.edu}, Denis Ovchinnikov\footnote{Stevens Institute of Technology, dovchinn@stevens.edu}
}

\begin{document}
\maketitle{}
\maketitle{}
\begin{abstract}
We solve the following algorithmic problems using $\tcz{}$ circuits, or in logspace and quasilinear time, 
uniformly
in the class of nilpotent groups with 
bounded nilpotency class and rank: 
subgroup conjugacy, computing the normalizer and isolator 
of a subgroup, coset intersection, and 
computing the torsion subgroup. 
Additionally, if any input words are provided 
in compressed form as straight-line programs or in Mal'cev 
coordinates the algorithms run in quartic time.
\end{abstract}

\tableofcontents

\section{Introduction}
This is the third paper in a series on complexity of algorithmic problems in finitely generated nilpotent groups. In the first paper \cite{MMNV2014}, we showed that the basic algorithmic problems (normal forms, conjugacy of elements, subgroup membership, centralizers, presentation of subgroups, etc.) can be solved by algorithms running in logarithmic space and quasilinear time. Further, 
if the problems are considered in `compressed' form with 
each input word provided as a straight-line program, 
we showed that the problems are solvable in polynomial time. The second paper \cite{MW17} pushed the complexity of these problems lower, showing that they 
may be solved by \tcz{} circuits.
Here we expand the list of algorithmic problems for 
nilpotent groups which may be solved in these low 
complexity conditions to include several fundamental problems concerning subgroups.

Note that in group theory algorithmic problems for subgroups of groups are usually much harder then the basic algorithmic problems mentioned above. Nevertheless, we present here algorithms for deciding the conjugacy of two subgroups of a finitely generated nilpotent group $G$, finding the normalizer and the isolator of a given subgroup of $G$, finding the torsion subgroup $T(G)$ of $G$, and finding the intersection of two cosets of subgroups of $G$, all of which may be implemented by \tcz{} circuits, or run  in logarithmic space and quasilinear time on a (multi-tape) Turing machine.
Furthermore, the compressed versions of these problems are solvable in polynomial (specifically, quartic) time.  All of the algorithms work uniformly over finitely generated nilpotent groups (i.e. the group may be included in the algorithm's input), however the complexity bounds depend on the nilpotency class and the rank (number of generators) of the presentation. When both are bounded, we solve all the problems uniformly in \tcz{} or logspace and quasilinear time.

Algorithmic problems in nilpotent groups have been studied for a long time.  On the one hand, it was shown that many of them are decidable and many sophisticated decision algorithms were designed (see, for example,  the pioneering paper \cite{KRRRC69} by Kargapolov et al. published in 1969 and the books \cite{Sim94} and \cite{HEO05} for more  recent techniques);  on the other hand,  there are some which have been known to be undecidable for some time (for instance,  decidability of equations \cite{Rom77}). 
Recent work by a variety of authors has introduced a host 
of decidable/undecidable problems.  New undecidable 
problems, including the knapsack problem, commutator and 
rectract problems are described in \cite{Loh15}, 
\cite{KLZ15}, \cite{MT16}, and \cite{Rom16}, while  positive 
decidability results for direct product decompositions and 
equations in the Heisenberg group are described in 
\cite{BMO16} and \cite{DLS15}. Decidability and 
undecidability results for equations over random 
nilpotent groups are also given in \cite{GMO16_Properties} 
and \cite{GMO16_Random}.

 However, it seems that this paper together with \cite{MMNV2014} and \cite{MW17} present the first thorough attempt to study the complexity of the problems, beyond the decidable/undecidable dichotomy.  In fact, it seems this is  currently the only known large class of non-abelian groups where the major algorithmic problems are shown to have low space and time complexity. 
  Another large class of such groups is, perhaps, the class of finitely generated free groups given by the standard presentations.  Even there,  if the free groups are given by arbitrary finite presentations the complexity of the algorithmic problems is still mostly unknown.


We have not yet mentioned one of the fundamental algorithmic 
problems in nilpotent groups: the isomorphism problem. 
It is decidable due to the  famous result of Grunewald and Segal \cite{GS80}. Nevertheless, not much is known about its complexity. 
\begin{problem}
Is the isomorphism problem in finitely generated nilpotent groups  decidable in polynomial time? Exponential time?
\end{problem}



\section{Background}
This section describes, summarizing from \cite{MMNV2014} and \cite{MW17}, 
how we will represent finitely generated nilpotent groups  (\S \ref{Sec:Malcev}) and their subgroups (\S \ref{Subsec:subgroups}), and gives black-box descriptions 
of several 
algorithms that we will be using as subroutines 
(\S \ref{Sec:FundamentalAlgorithms}).  We also give a brief 
introduction to the \tcz{} circuit model of computation, 
logspace computations, and 
the use of compressed words in algorithmic problems 
over groups (\S \ref{Sec:logspace}).

\subsection{Nilpotent presentations}
\label{Sec:Malcev}
Let $G$ be a finitely generated nilpotent group of nilpotency 
class $c$.  Then $G$ has lower central series 
\[
G=\Gamma_{1}\rhd \Gamma_{2}\rhd \ldots \rhd \Gamma_{c}\rhd \Gamma_{c+1}=1
\]
with $\Gamma_{i+1}=[G,\Gamma_{i}]$ for $i>1$. 
From this series we derive a presentation for $G$, 
as follows.

Each 
$\Gamma_{i}/\Gamma_{i+1}$ is a finitely generated abelian 
group.  We select and fix a finite generating set 
$a_{s_{i-1}+1}\Gamma_{i+1},\ldots,a_{s_{i}}\Gamma_{i+1}$ for  $\Gamma_{i}/\Gamma_{i+1}$ and put 
\[
A = \{a_{1},a_{2},\ldots,a_{m}\}.
\]
For each $j=1,\ldots,m$, if $s_{i-1}+1\leq j\leq s_{i}$,  
we denote by $e_{j}$ the order 
of $a_{j}\Gamma_{i+1}$ in $\Gamma_{i}/\Gamma_{i+1}$, using 
$e_{j}=\infty$ when the order is infinite.  Denote 
\[
\mathcal{T}=\{i \,|\, e_{i}<\infty\}.
\]
Provided that each generating set above  is chosen to correspond to a 
primary or invariant factor decomposition of 
$\Gamma_{i}/\Gamma_{i+1}$, every element $g\in G$ may be 
written uniquely in \emph{Mal'cev normal form} as 
\begin{equation}\label{MalcevNF}
g=a_{1}^{\alpha_{1}} a_{2}^{\alpha_{2}}\cdots a_{m}^{\alpha_{m}}
\end{equation}
where $\alpha_{i}\in\integers$ and if $i\in\mathcal{T}$ 
then $0\leq \alpha_{i}<e_{i}$.  The set $A$ is called a \emph{Mal'cev 
basis} of $G$ and the integers 
$(\alpha_{1},\ldots,\alpha_{m})$ are  the \emph{Mal'cev 
coordinates} of $g$.

For each $i=1,\ldots,m$, denote $G_{i}=\langle a_{i},\ldots,
a_{m}\rangle$.  An essential fact, which follows from 
the definition of the lower central series, is that 
for any $i<j$, 
\[
[a_{i},a_{j}]\in G_{\ell}
\]
for  some $\ell>j$.
From this it follows that relations of the form 
\begin{align}
a_{j}a_{i} & =   a_{i} a_{j}a_{\ell}^{\alpha_{ij\ell}}
a_{\ell+1}^{\alpha_{ij(\ell+1)}}
\cdots a_{m}^{\alpha_{ijm}} \label{relator1}\\
a_{j}^{-1}a_{i} & =  a_{i} a_{j}^{-1} a_{\ell}^{\beta_{ij\ell}}a_{\ell+1}^{\beta_{ij(\ell+1)}}\cdots a_{m}^{\beta_{ijm}}, \label{relator2}
\end{align}
with $\ell >j$, 
hold in $G$.  In addition, for each $i\in \mathcal{T}$ 
there is a relation of the form 
\begin{equation}\label{relator3}
a_{i}^{e_{i}} = a_{\ell}^{\mu_{i\ell}}a_{\ell+1}^{\mu_{i(\ell+1)}}\cdots a_{m}^{\mu_{im}}
\end{equation}
where $\ell>j$.  
The set  $\{a_{1},\ldots,a_{m}\}$, viewed as an abstract set 
of symbols, together with relators (\ref{relator1})--(\ref{relator3}) then 
form a presentation for $G$ called a 
\emph{nilpotent presentation}.  
In fact, any presentation of this form 
defines a nilpotent group.  Such a presentation 
is called \emph{consistent} if the order of each 
$a_{i}$ modulo $\langle a_{i+1},\ldots,a_{m}\rangle$ 
is precisely $e_{i}$. Note that $e_{i}=1$ is 
permitted in a nilpotent presentation.

For low-complexity algorithms, an essential 
property of nilpotent presentations is the following
(see \cite{MMNV2014} Thm. 2.3 and Lem. 2.5): if $w$ is any word over $A^{\pm}$, 
then the length of the Mal'cev normal form (\ref{MalcevNF}) 
of the element $g$ corresponding to $w$ in $G$ is bounded by a polynomial function of the length 
of $w$, with the degree of the polynomial depending on the nilpotency class 
$c$ and number of generators $r$ of $G$.  This fact plays a crucial role  in solving 
efficiently the 
fundamental algorithmic problems in finitely generated 
nilpotent groups.  


\subsection{Subgroups}\label{Subsec:subgroups}
All of our results concern subgroups of finitely generated 
nilpotent groups.  For every subgroup $H\leq G$ (all of 
which are, necessarily, finitely generated), one may define a 
unique generating set $(h_{1},\ldots,h_{s})$ called the \emph{full-form sequence} for $H$. The precise definition 
was given in \cite{Sim94} (and is reviewed in \cite{MMNV2014}), 
but we mention here only three facts about $(h_{1},\ldots,h_{s})$ that we will need.

First, let $B$ be the matrix in which row $i$ is the row 
vector consisting of 
the Mal'cev coordinates of $h_{i}$. Then $B$ is  in 
row echelon form and does not contain zero rows.  We denote by $\pi_{i}$ the pivot column 
of row $i$ of $B$.  Since this column corresponds to 
generator $a_{\pi_{i}}$, the Mal'cev normal form 
of $h_{i}$ begins with $a_{\pi_{i}}$, so 
$h_{i}\in G_{\pi_{i}}=\langle a_{\pi_{i}},\ldots,a_{m}\rangle$. 

Second, the number  of generators $s$ is bounded by the 
length $m$ of the Mal'cev basis.  Third, every element $h \in H$ can be uniquely presented in the form 
$$h = h_{1}^{\beta_{1}}\cdots h_{s}^{\beta_{s}},$$
where $\beta_{i}\in\integers$ and $0\leq \beta_{i}< e_{\pi_i}$ if $\pi_i\in\mathcal{T}$. Hence 

\[
H=\{h_{1}^{\beta_{1}}\cdots h_{s}^{\beta_{s}}\,|\, 
\beta_{i}\in\integers \;\mbox{and}\; 0\leq \beta_{i}< e_{\pi_i} \;
\mbox{if}\; \pi_i\in\mathcal{T}\}.
\]

\subsection{Logspace, \tcz, and compressed words}
\label{Sec:logspace}
Let $\mathcal{A}$ be a finite language.  
We are interested in both decision and search problems, and we 
may regard each such problem 
 as a function  $f:\mathcal{A}^{*}\rightarrow\mathcal{A}^{*}$.  The set $\mathcal{A}$ 
consists of a set of symbols, say $x_{1},\ldots,x_{n}$, which 
denote group generators, and a few extra symbols used to separate 
different parts of the input (commas to separate relators etc.).  
We will be computing $f(x)$ using \emph{logarithmic space} or using \emph{\tcz circuits}.  We recall both 
of these notions below.

\paragraph{Logspace.}
A \emph{$c$-logspace transducer}, where $c>0$ is a constant, is a multi-tape
Turing machine consisting of the following tapes: an `input' tape 
which is read-only, a constant number of read-write `work' tapes, and a 
write-only `output` tape.  For any input of length $L$, 
which is provided on the input tape, the amount of space the transducer 
is allowed to use on each work tape is $c\log(L)$.  The output of 
the machine is the content of the output tape.   A function $f$ 
is said to be \emph{logspace computable}, or more casually 
the associated 
problem is solvable in logarithmic space, if there exists a constant $c$ 
and a 
$c$-logspace transducer that produces $f(x)$ on the output tape 
for any input $x$ appearing on the input tape.

Though computation on a $c$-logspace transducer 
puts a bound only on space 
resources, a polynomial time bound of $O(L^{c})$ is 
forced by the fact 
that the machine may not enter the same configuration twice 
(otherwise it will loop infinitely) and the number of 
configurations is bounded by a polynomial function of 
the input length.  The degree $c$ may be very 
high, and for this reason it is also desirable to show directly 
that our 
algorithms run in low-degree polynomial time, in 
particular quasilinear time
(i.e. $O(L\log^{k}(L))$ for some constant $k$).

Most of our algorithms invoke other logspace algorithms 
as subroutines, and as such we need to compute compositions $f\circ g$ of  logspace computable functions.  A standard argument 
shows that $f\circ g$ is again logspace computable, but  
in computing $(f\circ g)(x)$ in this way, each symbol of $g(x)$ is recomputed each time it is needed in computation of $f$,  which may give a significant increase in time complexity.  However, if the output $g(x)$ is always of size 
$O(\log(L))$, one may simply compute $g(x)$ first, store the output on the work tape, 
and then proceed to compute $f(g(x))$.  This is the case in all of our algorithms, so in this case the time complexity of $g$ is simply added to the overall 
time complexity.

\paragraph{\tcz{} circuits.} 
A \emph{\tcz{} circuit} with 
$n$ inputs 
is a boolean circuit of constant depth using NOT gates 
and unbounded fan-in AND, OR, and MAJORITY gates, 
such that the total number of gates 
is bounded by a polynomial function of $n$.  
A MAJORITY gate outputs 1 when more than half of its 
inputs are 1. A function $f(x)$ is \emph{\tcz{}-computable}
(more casually, `an algorithm is in \tcz{}')
if for each $n$ there 
is a \tcz{} circuit $F_{n}$ with $n$ inputs which produces 
$f(x)$ on every input $x$ of length $n$.  Essential 
for our purposes is the fact that the composition 
of two \tcz{}-computable functions is again 
\tcz{}-computable.

Since 
this definition of being computable 
only asserts that such a family
$\{F_{n}\}_{n=1}^{\infty}$ of  circuits exists, 
one normally imposes in addition a \emph{uniformity} condition 
stating that each $F_{n}$ is constructible in some 
sense.  We will only be concerned here with 
standard notion of DLOGTIME uniformity, which asserts 
that there is a random-access Turing machine which 
decides in logarithmic time whether in circuit $F_{n}$ 
the output of gate number $i$ is connected to the input 
of gate $j$, and determines the types of gates 
$i$ and $j$.  We refer the reader to \cite{Vol99} for 
further details on \tcz{}.

To put our results in context, we remind the reader of the following inclusions of complexity classes:
\[
\mathrm{TC}^{0} \subseteq \mathrm{LOGSPACE} \subseteq \mathrm{P} \subseteq \mathrm{NP}.
\]
It is not known whether any of these inclusions is 
strict.  Though every \tcz{}-computable function 
is also logspace-computable and polynomial-time computable,
our algorithm descriptions also give direct proofs 
of membership in these classes.

\paragraph{Compressed words.}
We are also interested in algorithms that run efficiently 
when the input is given in compressed format. 
The use of Mal'cev coordinates provides a natural 
compression 
scheme for elements of $G$:  each $g\in G$ may be encoded 
by a tuple of integers (its Mal'cev coordinates) written in binary.  Notice that if the size $m$ of the Mal'cev basis 
is bounded, a normal form of length $n$ may be encoded by $O(\log n)$ bits.  Since every finitely generated nilpotent 
group has a Mal'cev basis, it is natural to consider 
algorithmic problems in which 
input words represented in this compact way.  Of course, 
such `compressed problems' are, in terms of computational 
complexity, more difficult than their uncompressed siblings.

Since we will consider uniform algorithms, in which 
a finitely generated nilpotent group $G$ is given by an 
\emph{arbitrary} presentation as part of the input, we also 
consider two other compression schemes which do not depend on a  the specification of a Mal'cev basis.  First, we 
may simply allow exponents to be encoded in binary.  In this 
scheme, a word is encoded as a product of tuples $(g,m)$, 
representing $g^{m}$, 
where $g$ is a group generator or, recursively, a word of this form, and $m$ 
is a binary integer.  For example, $(x^4 y^2)^8 x^{-6}$ is 
encoded as $(((x,0100)(y,0010)),1000)(x^{-1},0110)$.

Second, we consider straight-line programs, that is, 
context-free grammars that generate exactly one string. 
Formally, a 
\emph{straight-line program} or \emph{compressed word}
over an alphabet $\mathcal{A}$ consists of a set 
$\{A_{1},A_{2},\ldots,A_{n}\}$ called the \emph{non-terminal symbols}
and for each non-terminal symbol $A_{i}$ a 
\emph{production rule} either of the form $A_{i}\rightarrow A_{j} A_{k}$ 
with $j,k<i$, or of the form $A_{i}\rightarrow a$ where 
$a\in \mathcal{A}\cup \{\epsilon\}$ with $\epsilon$ denoting 
the empty word. The non-terminal $A_{n}$ is termed 
the \emph{root}, and one `expands' the compressed word 
by starting with the one-character word  $A_{n}$ and successively 
replacing any non-terminal with the right side of its 
production rule until only symbols from $\mathcal{A}$ 
remain. 
The number $n$ of non-terminal symbols is the size 
of the program.  Compression arises since a program 
of size $n$ may expand to a word of length $2^{n-1}$.
We refer the reader to the survey article \cite{Loh12} 
and the monograph \cite{Loh14} for further information 
on compressed words, or to the introduction of 
\cite{MMNV2014} for some brief remarks.

\subsection{Fundamental algorithms for nilpotent groups}
\label{Sec:FundamentalAlgorithms}

Throughout this paper, we make extensive use of algorithms 
described in \cite{MMNV2014} and \cite{MW17}.  
We give below a summary of 
some of the most heavily-used ones, and we will use the names 
listed here, in \textbf{bold text}, to refer to their use. 

\begin{itemize}
\item \textbf{Full-form Sequence}: Given 
$H\leq G$, compute the full-form generating sequence 
for $H$. 
\item \textbf{Membership}: Given 
$g\in G$ and $H\leq G$, determine if $g\in H$ and if so, 
compute the unique expression $g=h_{1}^{\alpha_{1}}\cdots h_{s}^{\alpha_{s}}$ where $(h_{1},\ldots,h_{s})$ is the 
full-form sequence for $H$.
\item \textbf{Subgroup Presentation:} Given 
$H\leq G$, compute a consistent nilpotent presentation 
for $H$.
\item \textbf{Conjugacy}: 
Given $g,h\in G$, produce $x\in G$ such that $g^x=h$ 
or determine that no such $x$ exists. 
\item \textbf{Centralizer}: Given 
$g\in G$, compute a generating set for the centralizer of 
$g$ in $G$.
\item \textbf{Kernel}: Given 
$K\leq G$ and $\phi: K\rightarrow G_{1}$, produce a 
generating set for the kernel of $\phi$. 
\item \textbf{Preimage}: Given 
$K\leq G$, $\phi: K\rightarrow G_{1}$, and $h\in G_{1}$
guaranteed to be in $\phi(K)$, produce $k\in K$ such 
that $\phi(k)=h$. 
\end{itemize}

We will need some further details regarding the input/output of 
these algorithms as well as their complexity.

\emph{Input.} 
In each algorithm, we fix in advance 
two integers $c$ and $r$.  The ambient nilpotent groups 
$G$ and $G_{1}$ are part of the input (thus the algorithms 
are `uniform') 
but must be of nilpotency class at most $c$ and be presented 
using at most $r$ generators for the complexity bounds given below 
to be valid.
Group 
elements are given as words over the generating set(s), 
subgroups are specified by finite 
generating sets, and $\phi$ is given by listing the elements 
$\phi(k)$ for each given generator $k$ of $K$.
The length $L$ of the input is the sum of the lengths of 
all relators in $G$ and $G_{1}$ plus the lengths of all 
input words.  


\emph{Output.} Each output word is given as a word over the 
original generating set except possibly in \textbf{Full-form 
sequence} and \textbf{Membership}.  In these cases, 
the algorithm converts to a nilpotent presentation of $G$, 
if one is not already provided, and provides the output words 
in the new generators (the isomorphism may also be provided, 
see Lemma \ref{Lem:WordConversion} below). 
In \textbf{Centralizer} and \textbf{Kernel}, 
if the original 
presentation of $G$ is already a nilpotent presentation, 
one may assume that the subgroup generating set in 
the output 
is the full-form sequence. 

In every case, the total length of each output word is bounded 
by a polynomial function of $L$ and the number of 
output words is bounded by a constant. Optionally, the output words 
may be given by their Mal'cev coordinates.

\emph{Complexity.} 
Each algorithm may be implemented on a logspace transducer, and if so runs in time quasilinear in $L$.  The proofs 
are given in \cite{MMNV2014}.  Alternatively, each 
problem may be solved using \tcz{} circuits, as proved  
in \cite{MW17}.

\emph{Compressed inputs.} Each algorithm may also be run 
`with compressed inputs'.  In this case, any input word 
(including group relators) may be provided by (binary) 
Mal'cev coordinates, words with binary exponents, 
or straight-line programs, as described in \S 
\ref{Sec:logspace}. We will measure 
the size of the input in terms of the number $n$ 
of input words and the maximum size $M$ of any single 
input word (in number of bits or 
number of non-terminal symbols).
The space 
complexity of each algorithm is then $O(M)$ (it 
does not depend on $n$) and the time 
complexity is $O(nM^{3})$.  All output is provided in 
the corresponding  
compressed format.  Although each input word, in its expanded 
form, may have length $O(2^{M})$, the polynomial bound 
for the length of output words implies that each output 
word, in expanded form, has length $O(2^{dM})$ where 
$d$ is the degree of the aformentioned polynomial bound. 
Since $d$ is constant, the compressed size of each output 
remains $O(\log(2^{dM}))=O(M)$.

\begin{remark}
We place no restriction on the number $n$ of input words.  
In all of the algorithms, any variable-sized set of input words 
(e.g. list of subgroup generators, group relators) will be fed as input 
to the matrix reduction algorithm described in 
Thm. 3.4 of \cite{MMNV2014} and processed in the `piecewise' 
manner described there, one word at a time.  After this, 
sets of words usually only appear as full-form sequences 
for subgroups, the number of which is always bounded 
by a constant.  The value $n$ contributes a linear factor 
to the time complexity of this algorithm (in both uncompressed 
and compressed cases), but does not contribute to the space complexity.
\end{remark}

While neither these algorithms nor the ones we describe 
in this paper require that the input groups $G$ and $G_{1}$ 
be given by a nilpotent
presentation, this form is used internally by all of 
the algorithms.  Converting to such a presentation 
is accomplished as follows.

\begin{lemma}\label{Lem:WordConversion}
Let $c$ and $r$ be fixed integers.  There is an 
algorithm that, given a finitely presented nilpotent 
group $G=\langle X\,|\,R\rangle$ of nilpotency class 
at most $c$ and with $|X|\leq r$, 
a finite set $Y\subset G$, 
and a word $h$ over $X^{\pm}$ guaranteed to be in the 
subgroup $H=\langle Y\rangle$, produces 
\begin{itemize}
\item a consistent nilpotent presentation 
$\langle Y'\,|\,S\rangle$ for $H$, in which binary numbers are 
used to encode exponents in the relators $S$,
\item a map $\phi:Y'\rightarrow (Y^{\pm 1})^{*}$ which 
extends to an isomorphism 
$\langle Y'\,|\,S\rangle\rightarrow H$, and
\item a binary integer 
tuple $h'$ giving the Mal'cev coordinates of $h$ relative 
to $Y'$.  
\end{itemize}
The algorithm runs 
in space logarithmic in the input length $L$ and time quasilinear 
in $L$, or in \tcz{}, and the (expanded) length of each output word is bounded by a polynomial function of $L$. If 
compressed inputs are used (in $R$, $Y$, or $h$), the space requirement is $O(M)$ and the 
time is $O(nM^{3})$, where $n$ is the total number of input 
words and $M$ bounds the size of any single input word.
\end{lemma}
\begin{proof}\emph{Algorithm.}
Begin by applying Prop. 5.1 of \cite{MMNV2014} (or Lem. 5 
of \cite{MW17} in the \tcz{} case) to 
compute a consistent nilpotent presentation 
$G=\langle X'\,|\, R'\rangle$.  Here $X\subset X'$,  
the inclusion $X\hookrightarrow X'$ induces an 
isomorphism $\langle X\,|\,R\rangle\simeq \langle X'\,|\, R'\rangle$, and each element of $X'\setminus X$ is a 
commutator in elements of $X$. Use \textbf{Subgroup Presentation} to 
compute a nilpotent presentation 
$\langle Y'\,|\,S\rangle$ for $H$.  
The generating set 
$Y'=\{h_{1},\ldots,h_{s}\}$ is precisely the full-form sequence 
for $H$.  The relators have the form (\ref{relator1})-(\ref{relator3}), and we encode the exponents appearing 
on the right sides in binary.
To obtain $\phi$, note that each element of $Y'$ has 
the form $x_{1}^{\alpha_{1}}\cdots x_{m}^{\alpha_{m}}$, 
where $X'=\{x_{1},\ldots,x_{m}\}$.  We replace each $x_{i}$ 
with its definition as a commutator of elements of $X$ and 
encode the exponents $\alpha_{i}$ using binary numbers.
Finally, use 
\textbf{Subgroup Membership} with input $h$ and 
$\{h_{1},\ldots,h_{s}\}$, which returns an expression 
$h=h_{1}^{\gamma_{1}}\cdots h_{s}^{\gamma_{s}}$, giving 
the Mal'cev coordinates $(\gamma_{1},\ldots,\gamma_{s})$.

\emph{Complexity.} Follows immediately from \cite{MMNV2014} and \cite{MW17}.
Note that $s$ is a constant depending on $c$ and $r$.
\end{proof}

We will often use this lemma in the case $Y=X$ 
to convert from an arbitrary presentation of $G$ to a 
nilpotent presentation.  In this case, 
we may assume the algorithm uses $Y'=X'\supset X$.  We 
convert all input words into their Mal'cev coordinates 
(relative to $X'$) at the same time, and perform 
further computations directly 
on the   
Mal'cev coordinates.

Using binary numbers in the output is necessary in order to obtain quasilinear time, since writing down a word 
in its expanded form takes as many steps as the length 
of the word itself, which in this case is only bounded 
by a polynomial function of $L$. 

\section{Algorithmic problems}

Before presenting the algorithms, let us 
make a few remarks regarding their complexity analysis. 
The analysis of most of the algorthims is 
similar, so we present here a general argument 
and fill 
in any additional details in the proof of each algorithm.

First, note that the nilpotency class $c$ and maximum 
number of generators $r$ of the input group(s) are 
constant.  All other constants are expressible in terms 
of $c$ and $r$.

At the beginning of each algorithm, we convert to a 
nilpotent presentation, if necessary, using 
Lemma \ref{Lem:WordConversion}.  We denote the resulting 
Mal'cev basis by 
\[
\{a_{1},\ldots,a_{m}\}.
\] 
Note that 
$m$ is constant.
Word lengths are unchanged during this conversion 
(see Lemma \ref{Lem:WordConversion}). 
We are guaranteed by \cite{MMNV2014} Thm. 2.3 
that a word of length $L$ has a Mal'cev form of length 
polynomial in $L$, hence its coordinates require 
$O(\log L)$ bits to record. 

Our algorithms generally consist of a sequence of 
subroutine calls, using the algorithms described 
in \S\ref{Sec:FundamentalAlgorithms} as well as those 
described in this section, with some minor additional 
processing.  The complexity bounds described in 
\S\ref{Sec:FundamentalAlgorithms} also apply to 
the algorithms we describe in this section, 
as we will see.
In all cases, we prove that 
the total number of subroutine calls and the total 
number of words that must be stored in memory at 
any given time is constant.  Consequently, the entire 
algorithm can, in principle, be expressed as a composition 
of a constant number of functions.  Each such 
function is \tcz{}-computable, hence so is the 
composition.
Note that to `store $x$ in memory' in \tcz{} terms 
means to add a parallel computation branch computing 
$x$.

Though it follows immediately that we have logspace solutions to these problems, we wish to prove that one may in fact run 
the algorithms on a logspace transducer in quasilinear 
time.  To do so, we must 
show that each subroutine may be run directly `in memory' 
on the logspace transducer.  

This is achieved by invoking each subroutine in its 
`compressed' form.  Initially, all input words are 
converted into $O(\log L)$-bit Mal'cev coordinate form.
In this process, any variable-sized set of words 
(subgroup generators or group relators) is reduced 
to a constant-sized set (the full-form sequence).  
This size is bounded by $m$, and we often assume it 
is precisely $m$ for notational convenience.
Each subroutine is then called with a constant number 
of $O(\log L)$-bit words. It will therefore run in space 
$O(\log L)$ and time $O(\log^{3} L)$, and 
produce a constant number of 
$O(\log n)$-bit output words.  

For compressed inputs, the argument is similar.  As 
we observed earlier, the polynomial length bound 
implies that the compressed size of words remains 
$O(M)$ throughout the algorithm. Each subroutine therefore 
has space complexity $O(M)$ and time complexity $O(M^{3})$, so the overall space and time complexities are $O(M)$ 
and $O(n M^3)$.

Finally, let us note that if we have a constant number 
elements $g_{1},\ldots,g_{t}$ in Mal'cev form we can, 
by \cite{MMNV2014} Lem. 2.10, compute the Mal'cev form 
of the product $g_{1}\cdots g_{t}$ within the space and 
time bounds specified above, in both compressed and 
uncompressed cases.  We use this 
without mention to maintain elements in coordinate form.

\subsection{Subgroup conjugacy and normalizers} 

In this section we give an algorithm to determine 
whether or not two subgroups of a nilpotent group are conjugate and if so 
to compute a conjugating element.  A natural by-product of 
this algorithm is the computation of subgroup normalizers.

We begin with a preliminary lemma solving 
the simultaneous conjugacy problem for tuples 
of commuting elements.  
In fact, commutation is not required, but we will 
obtain this stronger result (Theorem 
\ref{conjugation_of_tuples}) as a corollary of the 
more complicated coset intersection algorithm.

\begin{lemma}\label{conjugation_of_commuting_tuples}
Fix positive integers $c$, $r$, and $l$. 
There is an algorithm that, given 
a finitely generated nilpotent group $G=\langle X\,|\,R\rangle$
of nilpotency class at most $c$ with $|X|\leq r$ 
and 
two tuples of elements $(a_1,\ldots,a_l)$ and $(b_1,\ldots,b_l)$ such that $[a_i,a_j]=[b_i,b_j]=1$ for all $1\leq i,j\leq l$, decides if there exists $g\in G$
such that 
\[
a_i^g=b_i 
\]
for all $1\leq i\leq l$.
The algorithm produces $g$ if one exists, returns 
a generating set for the centralizer of $\{b_{1},\ldots,b_{l}\}$, and may be run in space 
logarithmic in the length $L$ of the input and time quasilinear in $L$, or in \tcz{}.  The length of 
each output word is 
bounded by a polynomial function of $L$.
If compressed inputs are used, the algorithm uses space 
$O(M)$ 
and time $O(n M^{3})$, where $n=|R|$ and $M$ bounds the 
encoded size of each input word.
\end{lemma}

\begin{proof}
\emph{Algorithm.}
If necessary, use \textbf{Lemma \ref{Lem:WordConversion}} to convert 
to a nilpotent presentation.  Next, 
we check conjugacy of $a_{1}$ with $b_{1}$ using 
the \textbf{Conjugacy Algorithm}. If they are not 
conjugate, we may return `No'.  Otherwise, we obtain $h$ 
such that $a_1^h=b_1$ and we 
compute a generating set for $C_{G}(b_{1})$ using the 
\textbf{Centralizer Algorithm}.

If $l>1$, we proceed recursively. Notice that $g$ 
exists if and only if there exists $x\in G$ such 
that $(a_{i}^{h})^{x}=b_{i}$ for $i=1,\ldots,l$, since 
we may put $x=h^{-1}g$. Further, such $x$ must lie 
in $C_{G}(b_{1})$ since $b_{1}=(a_{1}^{h})^{x}=b_{1}^{x}$.
Therefore it suffices to call \textbf{Lemma 
\ref{conjugation_of_commuting_tuples}} 
recursively with the (commuting) tuples 
$(a_{2}^{h},\ldots,a_{l}^{h})$ and 
$(b_{2},\ldots,b_{l})$ and the subgroup $C_{G}(b_{1})$ 
in place of $G$. 
Before making the recursive call, we use \textbf{Lemma \ref{Lem:WordConversion}} to convert to a nilpotent 
presentation for $C_{G}(b_{1})$ 
and write each of $a_{i}^{h}$, $b_{i}$ relative to this 
presentation.


If we obtain a conjugator $x$, we may return $g=h x$, using the 
map $\phi$ provided by Lemma \ref{Lem:WordConversion} to write $x$ in the original 
generators $X$.  In addition, we obtain a generating set for the centralizer of 
$\{b_{2},\ldots,b_{l}\}$ in $C_{G}(b_{1})$, which is precisely a generating set for the centralizer of the 
complete set $\{b_{1},\ldots,b_{l}\}$ in $G$.  As above, we must use $\phi$ to write 
these words in generators $X$.
If the recursive call returns `No', then the tuples are 
not conjugate.

\emph{Complexity.}
The depth $l$ of the recursion is constant and we need 
only store $h$ and the (constant-sized) generating set 
for the centralizer at each step of the recursion, hence the general 
argument given at the beginning of the section applies.


\end{proof}

We now give the algorithm for determining conjugacy 
of two subgroups.

\begin{theorem}\label{Thm:ConjugacySubgroups}\label{conjugation_of_subgroups}
Fix integers $c$ and $r$.
There is an algorithm that, given a finitely presented 
nilpotent group $G=\langle X\,|\,R\rangle$ of nilpotency 
class at most $c$ with $|X|\leq r$ and two subgroups $H$ and $K$, determines if 
there exists  $g\in G$ such that 
\[
H^g=K
\]
and 
if so finds such an element $g$ as well as 
\[
\mbox{a generating 
set for the normalizer 
$N_{G}(K)$}
\]
of $K$.  The algorithm 
runs in space logarithmic 
in the total length $L$ of the input
and time quasilinear in $L$, or in \tcz{}, and the length of every 
output word is bounded by 
a polynomial function of $L$.  If compressed inputs are used, the 
space complexity is $O(M)$ and the time complexity $O(nM^{3})$ where $n$ is the total number of input words 
and $M$ bounds the encoded size of each input word.
\end{theorem}


\begin{proof}
\emph{Algorithm.} Begin by converting, if necessary, 
to a nilpotent
presentation of $G$ using \textbf{Lemma 
\ref{Lem:WordConversion}}.

The algorithm recurses on the maximum $j$ 
such that $H\cap\Gamma_{j}\neq 1$ and $H\cap\Gamma_{j+1}=1$.
To find $j$, simply compute the 
\textbf{Full-form Sequence} 
for $H$ and observe that if the last element of the sequence 
begin with the letter $a_{k}$ then $j$ is the unique index 
such that 
$a_{k}\Gamma_{j+1}$ belongs to the generating set of
$\Gamma_{j}/\Gamma_{j+1}$ (see \S\ref{Sec:Malcev}).
Compute similarly 
the maximum $j'$ such that $K\cap\Gamma_{j'}\neq 1$ 
and $K\cap\Gamma_{j'+1}=1$.  If $j\neq j'$, then $H$ and $K$ 
are not conjugate since their conjugacy would imply conjugacy of 
$H\cap \Gamma_{i}$ with $K\cap \Gamma_{i}$ for all $i$ 
(since the $\Gamma_{i}$ are normal subgroups), hence 
equality of $j$ and $j'$. 

Denote $H_{j}=H\cap \Gamma_{j}$ and produce the 
full-form sequence for this group 
by taking the elements of the full-form sequence for $H$ that are in $\Gamma_{j}$.
Proceed similarly for $K_{j}=K\cap \Gamma_{j}$.  
Next, we check conjugacy of $H_{j}$ with $K_{j}$.

\emph{Conjugacy of $H_{j}$ with $K_{j}$.}
Let $\overline{\phantom{\phi}}:G\rightarrow G/{\Gamma_{j+1}}$
be the natural homomorphism. By the definition of central series, $G$ acts trivially by conjugation on $\overline{\Gamma_j}$.  Hence if $H_{j}$ and $K_{j}$ 
are conjugate then $\overline{H_{j}}=\overline{K_{j}}$. 
We first check if $\overline{H_{j}}=\overline{K_{j}}$, returning `No' if not.  To do so, it suffices to compute the \textbf{Full-form Sequences} for 
$\overline{H_{j}}$ and $\overline{K_{j}}$, and check them for equality.

Let $(h_1,\ldots,h_l)$ be the full-form sequence for $H_{j}$, 
computed above.
We now produce a generating set $(k_1,\ldots,k_l)$ for 
$K_{j}$ such that $\overline{h_{i}}=\overline{k_{i}}$ for 
all $i$, as follows. Use the \textbf{Preimage} 
algorithm, with the subgroup $K_{j}$, the homomorphism 
$\overline{\phantom{\phi}}:K_{j}\rightarrow G/\Gamma_{j+1}$, 
and the element $\overline{h_{i}}$, to produce each $k_{i}$. 
Since $K_{j}\cap\Gamma_{j+1}=1$, $(k_{1},\ldots,k_{l})$ 
generates $K$.

We claim for any $x\in G$, $H_{j}^x=K_{j}$ if and only if 
$h_{i}^x=k_{i}$ for $i=1,\ldots,l$. Indeed, since the tuples
$(h_{1},\ldots,h_{l})$ and $(k_{1},\ldots,k_{l})$ 
are generating sets their conjugacy implies $H_{j}$ and $K_{j}$ are 
conjugate. Conversely, if $H_{j}^x=K_{j}$ then $h_i^x\in K_{j}$ for all $i$.  But  $\overline{h_i^x}=\overline{h_i}=\overline{k_i}$, and since $\ \bar{ }\ $ is 
injective on $K_{j}$, we have $h_i^{x}=k_i$ for all $i$.
Also observe that $H_{j}$ is abelian, since  
\[
[H_{j},H_{j}]\leq H_{j}\cap\Gamma_{2j}\leq H_{j}\cap\Gamma_{j+1}=1,
\]
and 
 similary for $K_{j}$.  Hence $(h_{1},\ldots,h_{l})$ and 
 $(k_{1},\ldots,k_{l})$ are both 
 tuples of commuting elements. So to determine 
conjugacy of $H_{j}$ with $K_{j}$ it suffices to use the algorithm of 
\textbf{Lemma 
\ref{conjugation_of_commuting_tuples}} to determine 
conjugacy of $(h_{1},\ldots,h_{l})$ and $(k_{1},\ldots,k_{l})$ and if so find a conjugator $x$ and a generating set $Y$
for $C_{G}(K_{j})$.  In fact, $C_{G}(K_{j})=N_{G}(K_{j})$ 
since if any element $y\in G$ 
normalizes $K_{j}$, then for each $i=1,\ldots,l$ we have 
$k_{i}^y \in K_{j}$ and hence $k_{i}^y=k_{i}$, arguing as above. 

\emph{Recursion.} If $j=0$, then $H=H_{j}$ and $K=K_{j}$ and we have already solved 
the problem.  Otherwise, letting 
\[
\widehat{\phantom{\phi}}:N_G(K_{j})\rightarrow N_G(K_{j})/K_{j}
\]
be the canonical homomorphism,
we reduce the problem to 
conjugation of $\widehat{H^{x}}$ and $\widehat{K}$ in $N_{G}(K_{j})/K_{j}$, as 
follows.

An element $g$ such that $H^g=K$ exists if and only if there exists $y\in G$ 
such that $(H^{x})^{y}=K$. Such an element $y$ must lie in $N_{G}(K_{j})$, since 
\[
K_{j}^{y}=(H_{j}^{x})^{y}
    =(H\cap\Gamma_{j})^{xy} 
    =H^{xy}\cap \Gamma_{j}^{xy}
    =K\cap\Gamma_{j}=K_{j}.
\]
Now $K\leq N_{G}(K_{j})$, and $H^{x}\leq N_{G}(K_{j})$ since 
\[
K_{j}=H_{j}^{x}=(H\cap\Gamma_{j})^{x}=H^{x}\cap\Gamma_{j}\unlhd H^{x}.
\]
Finally, if 
$(\widehat{H^{x}})^{\widehat{y}}=\widehat{K}$ for some 
$\widehat{y}\in N_{G}(K_{j})/K_{j}$, we claim that $(H^{x})^{y}=K$.  Indeed, if $k\in K$ then for some 
$h\in H$ and $k'\in K_{j}$ we have 
$k=y^{-1}h^x y k'=y^{-1}(h^{x} (k')^{y^{-1}})y$. 
But $y\in N_{G}(K_{j})$ and  $K_{j}\leq H^{x}$, 
so $h^{x} (k')^{y^{-1}}\in H^{x}$
and the inclusion $K\subseteq H^{x}$ 
follows.  The reverse inclusion is proved similarly.

In order to solve the conjugation problem 
of $\widehat{H^{x}}$ and $\widehat{K}$ in 
$N_{G}(K_{j})/K_{j}$, we first use \textbf{Lemma \ref{Lem:WordConversion}}, 
with the generating set $Y$, 
to find a nilpotent presentation for $N_{G}(K_{j})$ and 
to convert the generating sets for $H^{x}$, $K$, and $K_{j}$ into coordinate 
form in this presentation. Add the generators of $K_{j}$ to this presentation 
to obtain a presentation for $N_{G}(K_{j})/K_{j}$, and 
call \textbf{Theorem 
\ref{Thm:ConjugacySubgroups}} with this presentation and the subgroups 
$\widehat{H^{x}}$ and $\widehat{K}$.

It is essential to prove that the value of $j$ decreases in the recursive call. 
Letting $N_{j}$ denote term $j$ of the lower central series of $N_{G}(K_{j})$, 
we have that $N_{j}\leq \Gamma_{j}$, hence $K\cap N_{j}\leq K\cap \Gamma_{j}=K_{j}$, 
and the intersection is trivial modulo $K_{j}$, hence $j$ must decrease.

The recursive call either proves that $\widehat{H^{x}}$ and $\widehat{K}$ are not 
conjugate, in which case $H$ and $K$ are not conjugate, or returns a conjugator 
$y K_{j}$ and a generating set 
$Z \cdot K_{j}$ for the normalizer of  
$\widehat{K}$ in $N_{G}(K_{j})/K_{j}$.  
Note that 
$y$ (and each element of $Z$) is 
given as a word over the generating set of 
$N_{G}(K_{j})$ with binary exponents.  We 
convert back to the generating set $X$ of $G$ 
using the map $\phi$ provided by Lemma \ref{Lem:WordConversion}. For the conjugator, 
we return the word $g=xy$.

For the normalizer, we append to $Z$ a generating set of $K_{j}$ to 
obtain a generating set $Z'$ for the normalizer of $K$ in $N_{G}(K_{j})$. But this is precisely the normalizer of $K$ 
in $G$: if $K^{z}=K$ for some $z\in G$ then 
$K_{j}^{z}=K^{z}\cap \Gamma_{j}^{z}=K\cap\Gamma_{j}=K_{j}$ 
and so $N_{G}(K)\leq N_{G}(K_{j})$.

\emph{Complexity.} 
The depth of the recursion is bounded by the constant $c$, 
and the number of words to store in memory is constant.
\end{proof}

It should be noted that while the algorithm does not 
compute the normalizer of $K$ in the event that $H$ and 
$K$ are not conjugate, one may of course obtain it by running 
the algorithm with $H=K$.

\subsection{Coset intersection}

We describe an  
algorithm to compute 
the intersection of cosets in finitely generated nilpotent groups, and apply 
it to solving the simultaneous conjugacy problem.
Recall that in any group, the intersection $g_{1} H\cap g_{2} K$ of two 
cosets is, if non-empty, a coset of the intersection $H\cap K$.

\begin{theorem}\label{intersection_of_cosets}
Fix integers $c$ and $r$.  There is an algorithm that, given 
a finitely presented nilpotent group 
$G=\langle X\,|\,R\rangle$ of nilpotency class at most $c$ 
with $|X|\leq r$, two subgroups 
$H$ and $K$ of $G$, and two elements $g_{1}$ and $g_{2}$ of $G$, 
determines if the intersection $g_{1} H\cap g_{2} K$ is 
non-empty and if so, produces a generating set for 
$H\cap K$ and an element $g'\in g_{1}H\cap g_{2}K$, hence 
\[
g_{1} H\cap g_{2} K= g'(H\cap K).
\]
The algorithm runs in space logarithmic in the length $L$ 
of the input and time quasilinear in $L$, or in \tcz{}.
If compressed inputs are used, the 
space complexity is $O(M)$ and the time complexity $O(nM^{3})$ where $n$ is the total number of input words 
and $M$ bounds the encoded size of each input word.
\end{theorem}

\begin{proof}
Begin by using \textbf{Lemma \ref{Lem:WordConversion}} 
to convert to a nilpotent presentation 
for $G$, if necessary. We proceed by induction on the nilpotency class 
$c$.

\emph{Base case c=1.}  In this case, $G$ is abelian. First, 
we will determine if the intersection 
is non-empty and if so find $g$.  Writing 
\[
g_{1} H\cap g_{2} K=g_{2}(g_{2}^{-1}g_{1} H\cap K),
\]
it suffices to determine if there exists $h\in H$ such that 
$g_{2}^{-1}g_{1} h\in K$. Since $G$ is abelian, this occurs if and 
only if $g_{2}^{-1}g_{1} \in \langle H\cup K\rangle$. We 
use the \textbf{Membership algorithm}, with 
the union of the \textbf{Full-form sequences} of 
$H$ and $K$ as a generating 
set for $\langle H\cup K\rangle$, to determine if this is the case, 
returning `No' if it is not.  Otherwise, we obtain 
an expression of $g_{2}^{-1} g_{1}$ as a linear combination 
of the elements of the full-form sequence for 
$\langle H\cup K\rangle$. We can convert to an expression 
in terms of the full-form sequences for $H$ and $K$,
thus obtaining an 
expression $g_{2}^{-1}g_{1}=hk$ for some 
elements 
$h\in H$ and $k\in K$, by following the procedure described 
in Cor. 3.9 of \cite{MMNV2014} (essentially, recording 
an expression of each matrix row in terms of the given 
generators during the matrix reduction process).  
This corollary gives only polynomial time, but Thm. 14 
of \cite{MW17} gives the corresponding result for 
\tcz{} (hence logspace), though we need the fact that 
$g_{2}^{-1} g_{1}$ and the full-form sequences 
of $H$ and $K$ are stored using only $O(\log L)$ bits.
We now set 
$g=g_{1}h^{-1}$ and obtain $g_{1} H\cap g_{2} K = g(H\cap K)$.

We will now find a generating set for $H\cap K$.
Let $\{u_{1},\ldots,u_{n}\}$ be the generating set for $H$ and  
consider the homomorphism 
$\phi:\integers^{n}\rightarrow G$ defined by
\[
\phi(\alpha_{1},\ldots,\alpha_{n})=u_{1}^{\alpha_{1}}\cdots
u_{n}^{\alpha_{n}}
\]
and the composition $\phi':\integers^{n}\rightarrow G\rightarrow G/K$.
An element  
$u_{1}^{\alpha_{1}}\cdots u_{n}^{\alpha_{n}}$ of $H$ 
is also an element of $K$ if and only if  
$(\alpha_{1},\ldots,\alpha_{n})$ is in the kernel of $\phi'$,
 hence $H\cap K=\phi(\ker \phi')$.
To compute the kernel, add the generators of $K$ to the relators of $G$ to obtain a presentation of $G/K$, and pass this 
group together with $\phi'$ and the standard presentation 
of $\integers^{n}$ to the \textbf{Kernel algorithm}.
Applying $\phi$ to each resulting subgroup generator, 
we obtain a generating set (in fact, the full-form sequence) 
for $H\cap K$.

\emph{Inductive case.} Denote by 
$\overline{\phantom{\phi}}:G\rightarrow G/\Gamma_{c}$ the canonical 
homomorphism.  Invoke \textbf{Theorem \ref{intersection_of_cosets}} recursively in $G/\Gamma_{c}$ 
with inputs $\overline{H}$, $\overline{K}$, $\overline{g_{1}}$, 
and $\overline{g_{2}}$. 
Note that it suffices to erase all generators of $\Gamma_{c}$ 
to compute $\overline{\phantom{\phi}}$ (more formally,
one may use Lemma \ref{Lem:WordConversion}).

If the recursive call determines that 
$\overline{g_{1}} \overline{H}\cap \overline{g_{2}}\overline{K}$ is empty, then 
so is $g_{1} H\cap g_{2}K$.  Otherwise, we obtain an element 
$\overline{\gamma}\in \overline{g_{1}}\overline{H}\cap\overline{g_{2}}\overline{K}$ 
and a generating set 
$\overline{w_{1}},\ldots,\overline{w_{l}}$ of 
$\overline{H}\cap\overline{K}$, hence 
\[
\overline{g_{1}}\overline{H}\cap \overline{g_{2}}\overline{K}= \overline{\gamma}\langle \overline{w_{1}},\ldots,\overline{w_{l}}\rangle.
\]
Denote by $\Lambda$ (but do not compute) the preimage 
of $\overline{H}\cap\overline{K}$ under 
$\overline{\phantom{\phi}}$. 
We will rewrite the intersection $g_{1}H\cap g_{2}K$ in the 
form 
\begin{equation}\label{int}
g_{1} H\cap g_{2} K = (g'(\Lambda\cap H))\cap (g' c_{0}(\Lambda\cap K))
\end{equation}
for certain $g'\in G$, $c_{0}\in \Gamma_{c}$ defined below. 
Compute a \textbf{Preimage} $x_{1}$ of $\overline{g_{1}^{-1}\gamma}$ 
in $H$ and a \textbf{Preimage} $x_{2}$ of $\overline{g_{2}^{-1}\gamma}$ in $K$.  Let 
\[
g'=g_{1}x_{1}.
\]
Since 
$\overline{g_{2} x_{2}}=\overline{\gamma}=\overline{g_{1}x_{1}}$, 
it follows that 
$(g_{1}x_{1})^{-1}g_{2}x_{2}=c_{0}\in \Gamma_{c}$.

To see that (\ref{int}) holds, let $u$ be an element of the left side. Then $u=g_{1} h$ and for some $h\in H$ and 
$\overline{u}=\overline{\gamma}\overline{\lambda}$ for some  
$\lambda\in\Lambda$. Then for some $c'\in\Gamma_{c}$, 
$u=g_{1}x_{1}\lambda c'$ 
hence $\lambda c'\in H$ since $x_{1}\in H$.
Clearly 
$\lambda c'\in\Lambda$, hence $u\in g'(\Lambda\cap H)$.  Similarly 
$u\in g'c_{0}(\Lambda\cap K)$.  Conversely, any element of the right 
side has the form $g'h=g_{1}x_{1}h$ for some $h\in H$ hence is in 
$g_{1}H$, and has the form $g'c_{0}k=g_{2}x_{2}k$ for some $k\in K$ hence 
is in $g_{2}K$.

We will now find the full-form sequences for 
$\Lambda\cap H$ and $\Lambda\cap K$.
Apply the \textbf{Preimage} algorithm 
to compute for each $\overline{w_{i}}$ preimages
$u_{i}'\in H$ and $v_{i}'\in K$. 
Compute a generating set $y_{1}',\ldots,y_{s'}'$
for $H\cap \Gamma_{c}$ by 
finding the \textbf{Full-form sequence} for $H$ 
and selecting only those elements that belong to $\Gamma_{c}$.
Similarly, compute a generating set $z_{1}',\ldots,z_{t'}'$ 
for $K\cap \Gamma_{c}$.  We now have 
\begin{eqnarray*}
\Lambda\cap H &=& 
\langle u_{1}',\ldots,u_{l}',y_{1}',\ldots,y_{s'}'\rangle.
\\
\Lambda\cap K &=& 
\langle v_{1}',\ldots,v_{l}',z_{1}',\ldots,z_{t'}'\rangle,
\end{eqnarray*}
Using the generating sets above, 
find the \textbf{Full-form sequence} 
$(v_{1},\ldots,v_{n},z_{1},\ldots,z_{s})$ for $\Lambda\cap K$, 
where $z_{1}$ denotes the first generator of the sequence 
that lies in $\Gamma_{c}$.  Likewise find the 
\textbf{Full-form sequence} 
$(u_{1},\ldots,u_{n'},y_{1},\ldots,y_{t})$ for $\Lambda\cap H$, with $y_{1}$ being the first generator in $\Gamma_{c}$.
Since $\Lambda\cap H$ and $\Lambda\cap K$ have the same 
image $\overline{H}\cap\overline{K}$ under $\overline{\phantom{\phi}}$, it follows that $n'=n$ and for 
all $i=1,\ldots,n$ that $u_{i}=v_{i}c_{i}$ for some $c_{i}\in\Gamma_{c}$.  We now have the full-form sequences 
\begin{eqnarray}
\Lambda\cap H &=& 
\langle v_{1}c_{1},\ldots,v_{n}c_{n},y_{1},\ldots,y_{t}\rangle.
\label{LH}
\\
\Lambda\cap K &=& 
\langle v_{1},\ldots,v_{n},z_{1},\ldots,z_{s}\rangle,\label{LK}
\end{eqnarray}

The next step produces a generating set of $H\cap K$ and 
the element $g$.  The correctness of this step is argued 
below.  
Denote $C_{1}=\langle c_{1},\ldots,c_{n},y_{1},\ldots,y_{t}\rangle$ 
and consider 
the intersection
\[
C_{1}\cap (\Lambda\cap K\cap \Gamma_{c})
\]
in 
the abelian group $\Gamma_{c}$. Define 
a 
homomorphism
$\psi:\integers^{n+t}\rightarrow \Gamma_{c}$
by
\[
\psi(\alpha_{1},\ldots,\alpha_{n},\beta_{1},\ldots,\beta_{t})=
c_{1}^{\alpha_{1}}\cdots c_{n}^{\alpha_{n}} y_{1}^{\beta_{1}}
\cdots y_{t}^{\beta_{t}}.
\]
Using the composition 
$\psi':\integers^{n+t}\rightarrow \Gamma_{c}\rightarrow \Gamma_{c}/(\Lambda\cap K\cap \Gamma_{c})$, we may then use the \textbf{Kernel} algorithm,
as in the base case, to  
produce a 
finitely generated subgroup 
$P=\langle p_{1},\ldots,p_{b}\rangle\leq \integers^{n+t}$ 
such that 
\[
C_{1}\cap \Lambda\cap K\cap \Gamma_{c}=\psi(P).
\]
The sequence $(p_{1},\ldots,p_{b})$ is the full-form sequence for $P$, so the corresponding matrix formed is in row-echelon form.  We denote 
$p_{i}=(p_{i1},\ldots,p_{i(n+t)})$ for $i=1,\ldots,b$.
In addition, we use the \textbf{Membership} algorithm, 
as described in the base case, to find 
\[
c_{0}'\in C_{1}\cap c_{0}(\Lambda\cap K\cap\Gamma_{c})
\]
if such an element exists and to write $c_{0}'$ in the 
form 
\[
c_{0}'=c_{1}^{\alpha_{1}'}\cdots c_{n}^{\alpha_{n}'}
y_{1}^{\beta_{1}'}\cdots y_{t}^{\beta_{t}'}.
\]
If $c_{0}'$ does not exist, we return `No'.  
Otherwise, we define 
\[
h=(v_{1}c_{1})^{\alpha_{1}'}\cdots (v_{n}c_{n})^{\alpha_{n}'}
y_{1}^{\beta_{1}'}\cdots y_{t}^{\beta_{t}'}
\]
and return $g=g'h$.  For the generating set of $H\cap K$, 
define the \emph{function} (it is not, in general, a 
homomorphism)
$\theta:\integers^{n+t}\rightarrow \Lambda\cap H$ by
\[
\theta(\alpha_{1},\ldots,\alpha_{n},\beta_{1},\ldots,\beta_{t})
=(v_{1}c_{1})^{\alpha_{1}}\cdots (v_{n}c_{n})^{\alpha_{n}}
y_{1}^{\beta_{1}}\cdots y_{t}^{\beta_{t}}
\]
and return the \textbf{Full-form sequence} 
for the subgroup generated by the set 
\[
\Pi=\{\theta(p_{i})\,|\, 1\leq i\leq b\}.
\]

It remains to prove the correctness of the last step.  First, 
we prove that $\Pi$ generates 
$H\cap K=(\Lambda\cap H)\cap (\Lambda\cap K)$. Take any 
$p_{i}=(\alpha_{1},\ldots,\alpha_{n},\beta_{1},\ldots,\beta_{t})$.  
Then, using the fact that $C_{1}\leq\Gamma_{c}$ is 
in the center of $G$,
\begin{align}
\theta(p_{i})&=(v_{1}c_{1})^{\alpha_{1}}
\cdots(v_{n}c_{n})^{\alpha_{n}} y_{1}^{\beta_{1}}\cdots
y_{t}^{\beta_{t}}\label{L1}\\
& =  v_{1}^{\alpha_{1}}
\cdots v_{n}^{\alpha_{n}} c_{1}^{\alpha_{1}}\cdots c_{n}^{\alpha_{n}} y_{1}^{\beta_{1}}\cdots
y_{t}^{\beta_{t}}.\label{L2}
\end{align}
Line (\ref{L1}) gives $\theta(p_{i})\in\Lambda\cap H$ and, 
since $c_{1}^{\alpha_{1}}\cdots c_{n}^{\alpha_{n}}y_{1}^{\beta_{1}}\cdots y_{t}^{\beta_{t}}=\psi(p_{i})\in \psi(P)\subset \Lambda\cap K$, line 
(\ref{L2}) gives $\theta(p_{i})\in \Lambda\cap K$. 
Hence $\langle \Pi\rangle \leq H\cap K$. 

For the opposite inclusion, let 
$(\Lambda\cap H)_{i}=\langle v_{i}c_{i},\ldots,v_{n}c_{n},y_{1},\ldots,y_{t}\rangle$ for
$1\leq i\leq n$ and $(\Lambda\cap H)_{n+1}=\langle y_{1},\ldots,y_{t}\rangle$.  We will prove, by induction 
on $i$ in the reverse order $n+1,\ldots,1$, that 
\[
(\Lambda\cap H)_{i}\cap(\Lambda\cap K)\leq \langle \Pi\rangle
\]
for all $i=n+1,\ldots,1$ (in particular for $i=1$).
For the base case $i=n+1$, let 
$q\in (\Lambda\cap H)_{n+1}\cap(\Lambda\cap K)$. 
Then $q=c_{1}^{0}\cdots c_{n}^{0}y_{1}^{\beta_{1}}\cdots y_{t}^{\beta_{t}}$ 
for some $(0,\ldots,0,\beta_{1},\ldots,\beta_{t})=p$. 
Since $q\in C_{1}\cap\Lambda\cap K$, 
we have $p\in P$. Since the matrix corresponding to $P$ 
is in row echelon form, we may write $p$ as a linear 
combination 
\[
p=\sum_{j=k}^{b}\gamma_{j}p_{j}
\]
where $p_{j}=(0,\ldots,0,p_{j(n+1)},\ldots,p_{j(n+t)})$ for all $j\geq k$.
Then 
\begin{align*}
q&=(y_{1}^{p_{k(n+1)}}\cdots y_{t}^{p_{k(n+t)}})^{\gamma_{k}}
\cdots (y_{1}^{p_{b(n+1)}}\cdots y_{t}^{p_{b(n+t)}})^{\gamma_{b}} \\
 &=\theta(p_{k})^{\gamma_{k}}\cdots \theta(p_{b})^{\gamma_{b}}
\end{align*}
hence $q\in\langle \Pi\rangle$.

For the inductive case, assume that 
$(\Lambda\cap H)_{i+1}\cap(\Lambda\cap K)\leq \langle \Pi\rangle$ for some $i+1\leq n+1$ and let 
$q\in (\Lambda\cap H)_{i}\cap (\Lambda\cap K)$. 
Then 
\[
q=(v_{1}c_{1})^{0}\cdots (v_{i-1}c_{i-1})^{0} (v_{i}c_{i})^{\alpha_{i}}\cdots (v_{n}c_{n})^{\alpha_{n}}
y_{1}^{\beta_{1}}\cdots y_{t}^{\beta_{t}}
\]
for some $(0,\ldots,0,\alpha_{i},\ldots,\alpha_{n},\beta_{1},\ldots,\beta_{t})=p'$.
Since $q\in \Lambda\cap K$ 
and $v_{i}^{\alpha_{i}}\cdots v_{n}^{\alpha_{n}}\in\Lambda\cap K$, it follows, rewriting $q$ as in (\ref{L2}), that 
$c_{i}^{\alpha_{i}}\cdots c_{n}^{\alpha_{n}}y_{1}^{\beta_{1}}\cdots y_{t}^{\beta_{t}}\in \Lambda\cap K$ and hence $p'\in P$. Hence 
\[
p=\sum_{j=k'}^{b}\gamma_{j}' p_{j}
\]
for some $\gamma_{j}'$ and $k'\geq 1$  
such that $p_{j}=(0,\ldots,0,p_{ji},\ldots,p_{j(n+t)})$ 
for all $j\geq k'$.
Now consider the element 
\[
q'=q \theta(p_{k'})^{-\gamma_{k'}'}\cdots \theta(p_{b})^{-\gamma_{b}'}.
\]
In the word $\theta(p_{k'})^{-\gamma_{k'}'}\cdots \theta(p_{b})^{-\gamma_{b}'}$, the generators 
$v_{1}c_{1},\ldots,v_{i-1}c_{i-1}$ do not appear. Further,
the total exponent sum of the generator $v_{i}c_{i}$ 
is $-\alpha_{i}$, while in $q$ it is $\alpha_{i}$. 
Since 
for any $i<l\leq n$ and $1\leq l'\leq t$ the commutators 
$[v_{i}c_{i},v_{l}c_{l}]$ and $[v_{i}c_{i},y_{l'}]$ 
are elements of $(\Lambda\cap H)_{i+1}$
we may collect all 
occurrences of $v_{i}c_{i}$ and eliminate its occurrence 
in $q'$. Hence 
$q'\in (\Lambda\cap H)_{i+1}\cap(\Lambda\cap K)$.
By induction, 
$q'\in \langle \Pi\rangle$ hence $q\in \langle \Pi\rangle$ 
as well. 

Regarding the intersection being non-empty, observe that by (\ref{int}), 
\begin{align*}
g_{1} H\cap g_{2} K\neq \emptyset & \iff 
(\Lambda\cap H)\cap c_{0}(\Lambda\cap K)\neq \emptyset.
\end{align*}
Now if $g=(v_{1}c_{1})^{\alpha_{1}}\cdots (v_{n}c_{n})^{\alpha_{n}}y_{1}^{\beta_{1}}\cdots y_{t}^{\beta_{t}}\in \Lambda\cap H$ is also an element of $c_{0}(\Lambda\cap K)$ then, again rewriting as in (\ref{L2}), 
$c_{1}^{\alpha_{1}}\cdots c_{n}^{\alpha_{n}}y_{1}^{\beta_{1}}\cdots y_{t}^{\beta_{t}}$ 
is an element of $C_{1}$ which also lies in $c_{0}(\Lambda\cap K\cap \Gamma_{c})$ hence $c_{0}'$ exists.  Conversely, if 
$c_{0}'$ exists, then we have  
\[
c_{0}'=c_{1}^{\alpha_{1}'}\cdots c_{n}^{\alpha_{n'}}y_{1}^{\beta_{1}'}\cdots y_{t}^{\beta_{t}'}=
c_{0} z_{1}^{\gamma_{1}'}\cdots z_{s}^{\gamma_{s}'}
\]
hence 
\[
(v_{1}c_{1})^{\alpha_{1}'}\cdots (v_{n}c_{n})^{\alpha_{n}'} y_{1}^{\beta_{1}'}\cdots y_{t}^{\beta_{t}'}=
c_{0} v_{1}^{\alpha_{1}'}\cdots v_{n}^{\alpha_{n}'} 
z_{1}^{\gamma_{1}'}\cdots z_{s}^{\gamma_{s}'}
\]
is an element of $(\Lambda \cap H)\cap c_{0}(\Lambda\cap K)$, 
hence this intersection is non-empty.  This proves the correctness of the decision problem.

Finally, we must show that $g\in g_{1}H\cap g_{2}K$. 
Since $h\in \Lambda\cap H$, 
we have $g=g'h\in g'(\Lambda\cap H)$. 
Since $c_{0}'\in c_{0} (\Lambda\cap K)$, we have 
$c_{0}'=c_{0} k$ for some $k\in \Lambda\cap K$ hence 
\[
g=g'h = g' v_{1}^{\alpha_{1}'}\cdots v_{n}^{\alpha_{n}'}c_{0}'
=g'c_{0} (kv_{1}^{\alpha_{1}}\cdots v_{n}^{\alpha_{n}})
\]
is an element of $g'c_{0} (\Lambda\cap K)$, as required.

\emph{Complexity.} The depth of the recursion is $c$, which 
is constant, so the total number of subroutine calls is  
constant.  The total number of group elements to record 
is also constant.
\end{proof}

As an application of the intersection algorithm, we may 
generalize Lemma \ref{conjugation_of_commuting_tuples} to 
solve the simultaneous conjugation problem for tuples 
in nilpotent groups.

\begin{theorem}\label{conjugation_of_tuples}
Fix integers $c$, $r$, and $l$. There is an algorithm that,
given 
a nilpotent group $G=\langle X\,|\,R\rangle$ of nilpotency class at most $c$ with $|X|\leq r$ and two 
tuples $(a_{1},\ldots,a_{l})$ and $(b_{1},\ldots,b_{l})$ 
of elements of $G$, computes:

\begin{itemize}
\item an element $g\in G$ such 
that 
\[
g^{-1}a_{i}g=b_{i}
\]
for $i=1,\ldots,l$ 
\item a generating set for the 
centralizer $C_{G}(b_{1},\ldots,b_{l})$,
or determines 
that no such element $g$ exists.
\end{itemize}
The algorithm runs 
in space logarithmic in the size $L$ of the input and 
time quasilinear in $L$, or in \tcz{}, and the length of each output 
word is bounded by a polynomial function of $L$.  If compressed inputs are used, the 
space complexity is $O(M)$ and the time complexity 
$O(nM^{3})$ where $n=|R|$  
and $M$ bounds the encoded size of each input word.
\end{theorem}
\begin{proof}
\emph{Algorithm.}
Begin by applying \textbf{Lemma \ref{Lem:WordConversion}} to 
convert to a nilpotent presentation if necessary.
Next, for each $i=1,\ldots,l$, use the \textbf{Conjugacy} algorithm to find $g_{i}\in G$ such that $a_{i}^{g_{i}}=b_{i}$.
If any pair is not conjugate, then $g$ does not exists and 
we may return `No'.  We also use the \textbf{Centralizer} algorithm to find, for each $i$, a generating 
set for $C_{G}(b_{i})$. 

Now for any $g\in G$ and any $i$, the equation 
\[
a_{i}^{g}=a_{i}^{g_{i}g_{i}^{-1}g}=b_{i}^{g_{i}^{-1}g}
\]
shows that $a_{i}^{g}=b_{i}$ if and only if 
$g_{i}^{-1}g\in C_{G}(b_{i})$, i.e. $g\in g_{i}C_{G}(b_{i})$. 
Hence the set of all possible conjugators is precisely the 
coset intersection $\bigcap_{i=1}^{l} g_{i} C_{G}(b_{i})$
which we may compute by iterating \textbf{Theorem \ref{intersection_of_cosets}}. As a by-product, we obtain 
a generating set for 
\[
\bigcap_{i=1}^{l} C_{G}(b_{i})=C_{G}(b_{1},\ldots,b_{l}). 
\]
\emph{Complexity.} Since $l$ is fixed, the number of subroutine calls and 
elements of $G$ to store is constant.
\end{proof}

\subsection{Torsion subgroup}
In every nilpotent group $G$ the set $T$ consisting of all 
elements of finite order forms a subgroup called 
the torsion subgroup.  We give 
an algorithm to compute, from a presentation of $G$, 
a generating set and presentation for $T$ as well as its 
order. 
We follow an algorithm outlined in 
\cite{KRRRC69}.


\begin{theorem}\label{Thm:Torsion}
Fix positive integers $c$ and $r$.  There is an algorithm that, given a finitely presented nilpotent group 
$G=\langle X\;|\; R\rangle$ 
of nilpotency class at most $c$ with $|X|\leq r$, produces 
\begin{itemize}
\item a generating set for the torsion subgroup $T$ of $G$,
\item a presentation for $T$, and  
\item the order of $T$.  
\end{itemize}
The algorithm runs in space logarithmic in the size $L$ of the given presentation and time quasilinear in $L$, 
or in \tcz{}.  The length 
of each output 
word is bounded by a polynomial function of $L$ and the number 
of such words is bounded by a constant.
If compressed inputs are used, the space complexity 
is $O(nM)$ and the time complexity is $O(nM^{3})$, where 
$n=|R|$ and $M$ bounds the length of each relator in $R$.
\end{theorem}
\begin{proof}
Define inductively a sequence $T_{1},T_{2},\ldots $ of finite 
normal subgroups of 
$G$ as follows.  Let $T_{1}=T(Z(G))$, which is clearly 
finite and normal.  For 
$i>1$ define the homomorphism 
$\phi_{i}: G\rightarrow G/T_{i-1}$ and set 
\[
T_{i} = \phi_{i}^{-1}\left(T\left( Z\left(G/T_{i-1}\right)\right)\right).
\]
Since $Z(G/T_{i-1})$ is abelian and finitely generated,
$T(Z(G/T_{i-1}))$ is finite 
and hence finiteness of $T_{i}$ follows from that of $T_{i-1}$.
Normality of $T_{i}$ follows from normality of $T_{i-1}$ 
in $G$ and of $T(Z(G/T_{i-1}))$ in $G/T_{i-1}$.
Since $G$ is Noetherian, the sequence must stabilize 
at some $T_{k}$.  But then $T(Z(G/T_{k}))$ is trivial, 
hence $G/T_{k}$ is torsion-free (its torsion subgroup must 
otherwise intersect its center), hence $T_{k}=T$.

\emph{Algorithm.} We compute the sequence described above. 
Begin by applying \textbf{Lemma \ref{Lem:WordConversion}} 
to compute 
a nilpotent presentation 
$G=\langle A\,|\,S\rangle$.
Since $Z(G)$ is simply the centralizer of any generating set, 
we may find the full-form sequence $(h_{1},\ldots,h_{m})$ for $Z(G)$ using 
\textbf{Theorem \ref{conjugation_of_tuples}} with the 
set $A$. Since $Z(G)$ is abelian, its torsion subgroup 
is generated by the set $X_{1}$ consisting of elements $h_{i}$ such that $i\in \mathcal{T}$.  Note that $\mathcal{T}$ is 
determined by examining the relators of the form (\ref{relator3}) is $S$.

Now assume, by induction, that 
we have a generating set $X_{i}$ for $T_{i}$. 
Use \textbf{Theorem \ref{conjugation_of_tuples}} 
with the nilpotent group 
$G/T_{i}=\langle A\,|\, S\cup X_{i}\rangle$ and the 
set $A$ to find, as described in the base case, the full-form 
sequence $(\tau_{1} T_{i},\ldots,\tau_{m}T_{i})$ for $T(Z(G/T_{i}))$. Then $X_{i}\cup \{\tau_{1},\ldots,\tau_{m}\}$ generates $T_{i+1}$, and we compute the the 
\textbf{Full-form sequence} $X_{i+1}$ of $T_{i}$. If $X_{i+1}=X_{i}$, 
then $T_{i+1}=T_{i}=T$. Otherwise,
we proceed with the next step of the induction.


Once we obtain the full-form sequence 
$(t_{1},\ldots,t_{n})$ for $T$, it suffices to run 
\textbf{Subgroup Presentation} to give a presentation 
for $T$.  Denote by $\pi_{1},\ldots,\pi_{n}$ the pivot columns 
of the matrix associated with $(t_{1},\ldots,t_{n})$ 
and by $\alpha_{ij}$ the $(i,j)$-entry of this matrix. 
Then every element of $T$ may be expressed 
uniquely in the form $t_{1}^{\beta_{1}}\cdots t_{n}^{\beta_{n}}$ where 
$0\leq \beta_{i}<e_{\pi_{i}}/\alpha_{i\pi_{i}}$, and every 
such expression gives a different element.  Hence 
the order of $T$ is 
\[
|T| = \prod_{i=1}^{n} e_{\pi_{i}}/\alpha_{i\pi_{i}}.
\]

\emph{Complexity.}
First, we will prove that the depth of the 
recursion is bounded by $c$. Let $Z_{i}=\{h\in G\,|\, [h,g]\in Z_{i-1} \;\mbox{for all $g\in G$}\}$
be the $i^{\mathrm{th}}$ term 
of the upper central series of $G$, with 
$Z_{1}=Z(G)$. We claim that 
\[
T_{i}\supseteq Z_{i}\cap T
\]
for all $i=1,\ldots,c$, hence $T_{c}\supseteq G\cap T$ 
so $T_{c}=T$ and the depth of the recursion is bounded by $c$.  
We proceed by induction.  For $i=1$ we have 
$T_{1}=T(Z(G))=Z_{1}\cap T$. Now let 
$g\in Z_{i}\cap T$ and consider $\phi_{i}(g)=\overline{g}$. 
Let $\overline{h}\in G/T_{i-1}$ and consider 
$[\overline{g},\overline{h}]=\overline{[g,h]}$.  Since $g\in T$, we have 
$[g,h]=g^{-1} g^h\in T$ and  since 
$g\in Z_{i}$, we have $[g,h]\in Z_{i-1}$. 
By the inductive assumption, $[g,h]\in T_{i-1}$ 
hence $[\overline{g},\overline{h}]=1$ 
hence $\overline{g}\in Z(G/T_{i-1})$.  Clearly 
$\overline{g}\in T(Z(G/T_{i-1}))$ hence $g\in T_{i}$, proving the 
claim.

Since the depth of the recursion is constant, the total 
number of subroutine calls is constant, as is the number 
of elements kept in memory, since each $X_{i}$ is a 
full-form sequence (hence of bounded length).
\end{proof}

In computing the order of $T$, recall that the numbers 
$e_{i}$ where $i\in \mathcal{T}$, appear as exponents 
in the nilpotent presentation 
computed by Lemma \ref{Lem:WordConversion}.  Consequently, 
each is bounded by a polynomial function of $L$.  Since 
the length $n$ of the full-form sequence for $T$ is bounded 
by a constant, the order of $T$ is polynomially bounded.

\begin{corollary}
If $G=\langle X\,|\, R\rangle$ is a nilpotent group 
of nilpotency class $c$ with $|X|\leq r$, then the 
order of the torsion subgroup of $G$ is bounded 
by a polynomial function of the sum of the lengths of 
the relators $R$.
\end{corollary}

\subsection{Isolator}

Recall that the \emph{isolator} of $H$ in $G$ is defined by 
\[
\mathrm{Is}_{G}(H)= \{ g\in G\,|\, g^{n}\in H \;\mbox{for some $n\neq 0$}\}
\]
and, in nilpotent groups, forms a subgroup.

\begin{theorem}
Fix integers $c$ and $r$.
There is an algorithm that, given a finitely presented 
nilpotent group 
$G=\langle X\,|\,R\rangle$ of nilpotency class at 
most $c$ with $|X|\leq r$, and a subgroup $H\leq G$, computes 
\[
\mbox{a generating set for the isolator 
$\mathrm{Is}_{G}(H)$.} 
\]
The algorithm runs in space logarithmic in the length $L$ 
of the input and time quasilinear in $L$, or in 
\tcz{}, and the length 
of each generator is bounded by a polynomial function 
of $L$. 
If compressed inputs are used, the 
space complexity is $O(M)$ and the time complexity $O(nM^{3})$ where $n$ is the number of input words   
and $M$ bounds the encoded size of each input word. 
\end{theorem}
\begin{proof}\emph{Algorithm.} 
First, apply \textbf{Lemma \ref{Lem:WordConversion}} 
to convert to a nilpotent presentation 
$G=\langle a_{1},\ldots,a_{m}\,|\, S\rangle$.

Let $N^{0}=H$ and for $i>0$ define 
$N^{i}=N_{G}(N^{i-1})$, the normalizer of $N^{i-1}$ in $G$. 
It is proved in \cite{KM79} Thm. 16.2.2 that $N^{c}=G$. 
Using \textbf{Theorem \ref{Thm:ConjugacySubgroups}}, 
we compute in turn the full-form generating sequences 
for each of 
the subgroups $H,N^{1},\ldots,N^{c}$ and using 
the \textbf{Subgroup Presentation} algorthim we 
compute a nilpotent presentation 
\[
N^{i}=\langle X_{i}\,|\, R_{i}\rangle
\]
for each.
We now proceed, by induction, to compute for each $i=0,\ldots,c$ a generating set $Y_{i}$ for 
$\mathrm{Is}_{N^{i}}(H)$. 
For $i=0$, we have $\mathrm{Is}_{N^{0}}(H)=H$ and we use the 
computed full-form sequence $X_{0}$ for $H$.
Now assume that we have a generating set $Y_{i-1}$ for $\mathrm{Is}_{N^{i-1}}(H)$.  

The subgroup $N^{i-1}$ is normal in $N^{i}$, and we 
will find the torsion subgroup of $N^{i}/N^{i-1}$. 
Using \textbf{Lemma \ref{Lem:WordConversion}}, write 
each element of $X_{i-1}$ in its $X_{i}$-coordinates.  Appending 
these elements to $R_{i}$ we obtain a presentation of 
$N^{i}/N^{i-1}$, which we pass to 
\textbf{Theorem \ref{Thm:Torsion}} to obtain a generating 
set  
$\{\tau_{1} N^{i-1},\ldots,\tau_{m}N^{i-1}\}$ for the torsion subgroup.  Then $\mathrm{Is}_{N^{i}}(N^{i-1})$ is generated 
by $X_{i-1}\cup\{\tau_{1},\ldots,\tau_{r}\}$.  
Converting these elements back to generators of $G$, we 
then compute 
the \textbf{Full-form sequence} $Z_{i}$ for $\mathrm{Is}_{N^{i}}(N^{i-1})$ and, using 
\textbf{Lemma \ref{Lem:WordConversion}} the corresponding 
nilpotent presentation 
\[
\mathrm{Is}_{N^{i}}(N^{i-1})=\langle Z_{i}\,|\, S_{i}\rangle.
\]

We claim that $\mathrm{Is}_{N^{i-1}}(H) \trianglelefteq \mathrm{Is}_{N^{i}}(N^{i-1})$.  Indeed, the property that $g^{n}\in H$ for some $n>0$ 
is unchanged under conjugation, and since $N^{i}$ normalizes $N^{i-1}$ all conjugates remain in $N^{i-1}$.  
Using \textbf{Lemma \ref{Lem:WordConversion}}, write 
each element of $Y_{i-1}$ in terms of the generators $Z_{i}$ 
and append these words to $S_{i}$ to obtain a presentation 
of 
$\mathrm{Is}_{N^{i}}(N^{i-1})/\mathrm{Is}_{N^{i-1}}(H)$. 
Now apply \textbf{Theorem \ref{Thm:Torsion}} to 
compute a generating set  
$\{\sigma_{1} \mathrm{Is}_{N^{i-1}}(H),\ldots,\sigma_{r} \mathrm{Is}_{N^{i-1}}(H)\}$ 
for the torsion subgroup.

Set $Y_{i}=\{\sigma_{1},\ldots,\sigma_{r}\}\cup Y_{i-1}$,
using the two prior calls to Lemma \ref{Lem:WordConversion} 
to write 
each $\sigma_{i}$ in generators of $G$.  We claim 
that $Y_{i}$ 
generates $\mathrm{Is}_{N^{i}}(H)$.  Clearly 
$Y_{i-1}\subset \mathrm{Is}_{N^{i-1}}(H)\subset\mathrm{Is}_{N^{i}}(H)$.  For each $\sigma_{i}$ there exists $p_{i}>0$ such that 
$\sigma_{i}^{p_{i}}\in \mathrm{Is}_{N^{i-1}}(H)$, therefore there exists $q_{i}>0$ such that $\sigma_{i}^{p_{i}q_{i}}\in H$.  
Hence $Y_{i}\subset \mathrm{Is}_{N^{i}}(H)$.
Now if $g\in \mathrm{Is}_{N^{i}}(H)$ then $g^{n}\in H$ for some $n>0$ and $g\in\mathrm{Is}_{H_{i}}(H_{i-1})$, hence 
$g \cdot\mathrm{Is}_{H_{i-1}}(H)$ lies in the torsion subgroup of $\mathrm{Is}_{H_{i}}(H_{i-1})/\mathrm{Is}_{H_{i-1}}(H)$, and so   
$g$ is an element of the subgroup generated by $Y_{i}$. 

Finally, return the \textbf{Full-form sequence} for 
$Y_{n}$.

\emph{Complexity.} Since the number of elements in a 
full-form sequence is bounded by $m$, the total number 
of elements in the sequences for $N^{i}$, $i=0,\ldots,c$, 
is constant, as is the number of relators in each $R_{i}$ 
and the number of elements in the generating set 
for the torsion subgroups.  The depth of the recursion 
is bounded by $c$.  Hence the total number of elements to 
store and the number of subroutine calls is constant.
\end{proof}

\bibliographystyle{alpha}
\bibliography{nilpotent_subgroup}

\newcommand{\etalchar}[1]{$^{#1}$}
\def\cprime{$'$}
\begin{thebibliography}{MMNV15}

\bibitem[BMO16]{BMO16}
Gilbert Baumslag, Charles~F. Miller, and Gretchen Ostheimer.
\newblock Decomposability of finitely generated torsion-free nilpotent groups.
\newblock {\em Internat. J. Algebra Comput.}, 26(8):1529--1546, 2016.

\bibitem[DLS15]{DLS15}
Moon Duchin, Hao Liang, and Michael Shapiro.
\newblock Equations in nilpotent groups.
\newblock {\em Proc. Amer. Math. Soc.}, 143(11):4723--4731, 2015.

\bibitem[GMO16a]{GMO16_Properties}
Albert Garreta, Alexei Miasnikov, and Denis Ovchinnikov.
\newblock Properties of random nilpotent groups.
\newblock 2016.
\newblock Preprint. arXiv:1612.01242 [math.GR].

\bibitem[GMO16b]{GMO16_Random}
Albert Garreta, Alexei Miasnikov, and Denis Ovchinnikov.
\newblock Random nilpotent groups, polycyclic presentations, and diophantine
  problems.
\newblock 2016.
\newblock arXiv:1612.02651 [math.GR].

\bibitem[GS80]{GS80}
Fritz Grunewald and Daniel Segal.
\newblock Some general algorithms. {II}. {N}ilpotent groups.
\newblock {\em Ann. of Math. (2)}, 112(3):585--617, 1980.

\bibitem[HEO05]{HEO05}
Derek~F. Holt, Bettina Eick, and Eamonn~A. O'Brien.
\newblock {\em Handbook of computational group theory}.
\newblock Discrete Mathematics and its Applications (Boca Raton). Chapman \&
  Hall/CRC, Boca Raton, FL, 2005.

\bibitem[KLZ15]{KLZ15}
Daniel K\"onig, Markus Lohrey, and Georg Zetzsche.
\newblock Knapsack and subset sum problems in nilpotent, polycyclic, and
  co-context-free groups.
\newblock 2015.
\newblock arXiv:1507.05145 [math.GR].

\bibitem[KM79]{KM79}
M.~I. Kargapolov and Ju.~I. Merzljakov.
\newblock {\em Fundamentals of the theory of groups}, volume~62 of {\em
  Graduate Texts in Mathematics}.
\newblock Springer-Verlag, New York-Berlin, 1979.
\newblock Translated from the second Russian edition by Robert G. Burns.

\bibitem[KRR{\etalchar{+}}69]{KRRRC69}
M.~I. Kargapolov, V.~N. Remeslennikov, N.~S. Romanovski{\u\i}, V.~A.
  Roman{\cprime}kov, and V.~A. {\v{C}}urkin.
\newblock Algorithmic questions for {$\sigma $}-powered groups.
\newblock {\em Algebra i Logika}, 8:643--659, 1969.

\bibitem[Loh12]{Loh12}
Markus Lohrey.
\newblock Algorithms on {SLP}-compressed strings: a survey.
\newblock {\em Groups Complex. Cryptol.}, 4(2):241--299, 2012.

\bibitem[Loh14]{Loh14}
Markus Lohrey.
\newblock {\em The compressed word problem for groups}.
\newblock SpringerBriefs in Mathematics. Springer, New York, 2014.

\bibitem[Loh15]{Loh15}
Markus Lohrey.
\newblock Rational subsets of unitriangular groups.
\newblock {\em Internat. J. Algebra Comput.}, 25(1-2):113--121, 2015.

\bibitem[MMNV15]{MMNV2014}
Jeremy Macdonald, Alexei Miasnikov, Andrey Nikolaev, and Svetla Vassileva.
\newblock Logspace and compressed-word computations in nilpotent groups.
\newblock 2015.
\newblock Preprint. arXiv:1503.03888 [math.GR].

\bibitem[MT16]{MT16}
Alexei Mishchenko and Alexander Treier.
\newblock Knapsack problem for nilpotent groups.
\newblock 2016.
\newblock Preprint. arXiv:1606.08584 [math.GR].

\bibitem[MW17]{MW17}
Alexei Miasnikov and Armin Wei{\ss}.
\newblock $\mathrm{TC^{0}}$ circuits for algorithmic problems in nilpotent
  groups.
\newblock 2017.
\newblock To appear in MFCS 2017 proceedings. Preprint at arXiv:1702.06616
  [math.GR].

\bibitem[Rom77]{Rom77}
V.~A. Roman'kov.
\newblock Unsolvability of the problem of endomorphic reducibility in free
  nilpotent groups and in free rings.
\newblock {\em Algebra i Logika}, 16(4):457--471, 494, 1977.

\bibitem[Rom16]{Rom16}
Vitaly~A. Roman'kov.
\newblock Diophantine questions in the class of finitely generated nilpotent
  groups.
\newblock {\em J. Group Theory}, 19(3):497--514, 2016.

\bibitem[Sim94]{Sim94}
Charles~C. Sims.
\newblock {\em Computation with finitely presented groups}, volume~48 of {\em
  Encyclopedia of Mathematics and its Applications}.
\newblock Cambridge University Press, Cambridge, 1994.

\bibitem[Vol99]{Vol99}
Heribert Vollmer.
\newblock {\em Introduction to circuit complexity}.
\newblock Texts in Theoretical Computer Science. An EATCS Series.
  Springer-Verlag, Berlin, 1999.
\newblock A uniform approach.

\end{thebibliography}

\end{document}